\date{\today}
\documentclass[10pt]{amsart}
\usepackage{amssymb,amsfonts,amsthm,amsmath,verbatim,lscape,color,tensor,url, todonotes}
\usepackage[margin=1.5in]{geometry}
\usepackage[enableskew]{youngtab}
\usepackage[all]{xy}
\usepackage[nameinlink,capitalise,noabbrev]{cleveref}

\newtheorem{thm}[subsection]{Theorem}
\newtheorem{prop}[subsection]{Proposition}
\newtheorem{cor}[subsection]{Corollary}
\newtheorem{lemma}[subsection]{Lemma}

\theoremstyle{definition}
\newtheorem{definition}[subsection]{Definition}

\newtheorem{example}[subsection]{Example}

\newtheorem{remark}[subsection]{Remark}
\numberwithin{equation}{section}

\newcommand{\powser}[1]{[\![#1]\!]}

\newcommand{\G}{\mathbb{G}}
\newcommand{\F}{\mathbb{F}}
\newcommand{\Sect}{\mathcal{O}}
\newcommand{\cL}{\mathcal{L}}
\newcommand{\Q}{\mathbb{Q}}

\newcommand{\Z}{\mathbb{Z}}

\newcommand{\QZ}[1]{(\Q_p/\Z_p)^{#1}}

\newcommand{\al}{\alpha}

\newcommand{\lra}[1]{\overset{#1}{\longrightarrow}}

\newcommand{\Prod}[1]{\underset{#1}{\prod}}

\newcommand{\Coprod}[1]{\underset{#1}{\coprod}}

\newcommand{\Loops}{\mathcal{L}}
\def \mmod{/\mkern-3mu /}

\DeclareMathOperator{\Aut}{Aut}

\DeclareMathOperator{\im}{im}

\DeclareMathOperator{\Hom}{Hom}
\DeclareMathOperator{\colim}{colim}

\DeclareMathOperator{\Spf}{Spf}

\DeclareMathOperator{\Sub}{Sub}

\DeclareMathOperator{\Level}{Level}

\DeclareMathOperator{\Tr}{Tr}
\DeclareMathOperator{\GL}{GL}
\DeclareMathOperator{\Cl}{Cl}
\DeclareMathOperator{\tors}{free}
\DeclareMathOperator{\Map}{Map}
\DeclareMathOperator{\rk}{rk}

\DeclareMathOperator{\lat}{\mathbb{L}}

\DeclareMathOperator{\cF}{\mathcal{F}}
\DeclareMathOperator{\qz}{\mathbb{T}}

\newcommand{\upperRomannumeral}[1]{\uppercase\expandafter{\romannumeral#1}}

\begin{document}
\title[Level structures and Morava $E$-theory]{Level structures on $p$-divisible groups from the Morava $E$-theory of abelian groups}

\author{Zhen Huan}
\author{Nathaniel Stapleton}

\begin{abstract}    % type your abstract below
The close relationship between the scheme of level structures on the universal deformation of a formal group and the Morava $E$-cohomology of finite abelian groups has played an important role in the study of power operations for Morava $E$-theory. The goal of this paper is to explore the relationship between level structures on the $p$-divisible group given by the trivial extension of the universal deformation by a constant $p$-divisible group and the Morava $E$-cohomology of the iterated free loop space of the classifying space of a finite abelian group.
\end{abstract}

\maketitle

%%%%%%%%%%%%%%%%%%%%   Start of main body of article

\section{Introduction}

Power operations for Morava $E$-theory have been studied for more than three decades. Due to the close connection between Morava $E$-theory and the arithmetic geometry of the universal deformation $\G$, the best strategy for understanding these operations has been to provide an algebro-geometric description of these maps whenever possible. In order to aid this endeavor, Strickland proved that, after taking the quotient by a transfer ideal, the $E$-cohomology of the symmetric group is the ring of functions on the scheme that represents subgroup schemes of a particular order in the universal deformation formal group \cite{etheorysym}. In \cite{ahs}, Ando, Hopkins, and Strickland proved that the additive power operation
\[
P_m/I_{Tr} \colon E^0 \rightarrow E^0(B\Sigma_m)/I_{Tr}
\]
is the ring of functions on the map of formal schemes that takes a deformation with a choice of subgroup of order $m$ to the quotient deformation - a canonical deformation with formal group given by the quotient. The target of the power operation can be simplified by making use of the fact that there is an injection
\[
E^0(B\Sigma_m)/I_{Tr} \hookrightarrow \Prod{A \subset \Sigma_m \text{transitive}} E^0(BA)/I_{Tr},
\]
where $I_{Tr} \subset E^0(BA)$ is another transfer ideal. The $E^0$-algebra $E^0(BA)/I_{Tr}$ is very closely related to the scheme of $A^*$-level structures on $\G$. This scheme was introduced by Drinfeld in \cite[Section 4]{drinfeldell1} and plays an important role in arithmetic geometry. These $E^0$-algebras are complicated but somewhat accessible and quite well-behaved.

Any kind of explicit calculation of these power operations above height $1$ turns out to be quite difficult. Explicit calculations were initiated by Rezk in \cite{rezkpowercalc} and have been successfully continued by Zhu in \cite{zhupo}. Certain variants of these power operations are more geometric and thus more amenable to calculation. In particular, work of Ganter \cite{ganterpowerops} and the first author \cite{Huan1, Huan2} describe and study the power operations on ``Tate $K$-theory" and Quasi-elliptic cohomology, respectively, completely geometrically. These cohomology theories are both closely related to height $2$ Morava $E$-theory. In fact, the character maps of the second author \cite{tgcm, ttcm} and of the second author with Barthel \cite{bscentralizers} can be used to approximate Morava $E$-theory by a variety of types of extensions of $p$-adic $K$-theory by the free loop space functor that are essentially coarse versions of Tate $K$-theory. As more variants of ``cohomology of the free loop space" arise in nature (eg. \cite{berwickevans2015twisted, Bismut, ganterpowerops, han2007supersymmetric, Huan2, noohi2019twisted, ttcm} and especially \cite{lurE}), it is worth gaining a better understanding of the target of the power operation in the case of the simplest extensions of Morava $E$-theory.

In this paper, we restrict our attention to the extension of Morava $E$-theory of the form $E^0(\Loops^h(-))$. That is, we just compose $E$-cohomology with the $h$-fold free loop space functor $\Loops^h(-)$. Although this ``cohomology theory" does not seem to inherit power operations from Morava $E$-theory, the extensions of $E$-theory that are expected to have power operations will map to it. Thus the first $E^0$-algebra that one might want to study in this context is $E^0(\Loops^hB\Sigma_m)/I_{Tr}$. This was accomplished in \cite{SSST}, which provides an algebro-geometric description in terms of subgroup schemes of order $m$ in the $p$-divisible group $\G \oplus (\Q_p/\Z_p)^h$, generalizing Strickland's fundamental result in \cite{etheorysym}. As level structures are more often easier to work with than subgroups, and also more amenable to calculation, in this paper we study $E^0(\Loops^{h}BA)/I_{Tr}$ for $A$ finite abelian. We prove a variety of results analogous to more classical results regarding $E^0(BA)/I_{Tr}$. Among other things, we describe the product decomposition of $E^0(\Loops^{h}BA)/I_{Tr}$ by making use of certain families of subgroups of $A$, give an algebro-geometric description of these $E^0$-algebras in terms of level structures on $\G \oplus (\Q_p/\Z_p)^h$, and describe the relation to $E^0(\Loops^hB\Sigma_m)/I_{Tr}$.

To state the main result more precisely, we need some setup. Assume that $A$ is a finite abelian $p$-group, let $\lat' = \Z_{p}^h$, and let $f \colon \lat' \to A$ be a continuous map of abelian groups. Let $\cF_f$ be the minimal family of subgroups of $A$ containing
\[
\{H \subset A | H \text{ is proper and } \im(f) \subseteq H\}.
\]
Associated to $\cF_f$ is the transfer ideal $I_{\cF_f} \subseteq E^0(BA)$, which is generated by the image of the transfers along the subgroups $H \in \cF_f$. We produce a factorization of $E^0$-algebras
\[
E^0(\Loops^{h}BA)/I_{Tr} \cong \Prod{f \in \hom(\lat',A)} E^0(BA)/I_{\cF_f}.
\]

The Pontryagin dual of $\lat'$ is $\qz' \cong (\Q_p/\Z_p)^h$ and the Pontryagin dual of $f$ is
\[
f^* \colon A^* \to \qz'.
\]
Using the map $f^*$, we define a functor from complete local $E^0$-algebras to sets called 
\[
\Level_{f^*}(A^*,\G \oplus \qz') \colon \text{complete local $E^0$-algebras} \to \text{Set}. 
\]
For a complete local $E^0$-algebra $R$, we define $\Level_{f^*}(A^*,\G \oplus \qz')(R)$ to be the set of homomorphisms $A^* \to \G(R) \oplus \qz'$ satisfying two properties: 
\begin{itemize}
\item the composite $ A^* \to \G(R) \oplus \qz' \to \qz'$ is equal to $f^*$, and 
\item the induced map $\ker(f^*) \to \G(R)$ is a level structure. 
\end{itemize}

By construction, there is a pullback of formal schemes
\begin{equation}
\xymatrix{\Level_{f^*}(A^*,\G \oplus \qz') \ar[r] \ar[d] & \Level(\ker(f^*), \G) \ar[d] \\ \Hom(A^*, \G) \ar[r] & \Hom(\ker(f^*), \G),}
\label{levelfdecomp}\end{equation}
where the left vertical map is induced by the projection $\G \oplus \qz' \to \G$.
%Think about if this is how we want to define this here in the intro. I'm not sure.

\begin{thm} \label{thmiso} There is a canonical isomorphism of $E^0$-algebras
\[
(E^0(BA)/I_{\cF_f})^{\tors} \cong \mathcal{O}_{\Level_{f^*}(A^*, \G\oplus \qz')},
\]
where $\mathcal{O}_{\Level_{f^*}(A^*, \G \oplus \qz')}$ denotes the ring of functions on $\Level_{f^*}(A^*, \G \oplus \qz')$ and $(E^0(BA)/I_{\cF_f})^{\tors}$ denotes the torsion-free part of the ring $E^0(BA)/I_{\cF_f}$.
\end{thm}

This result is proved by first using Hopkins--Kuhn--Ravenel character theory \cite{hkr} to produce a canonical isomorphism of $E^0$-algebras
\[
(E^0(\Loops^hBA)/I_{Tr})^{\tors} \cong \mathcal{O}_{\Level(A^*, \G\oplus \qz')},
\]
where $I_{Tr}$ is the ideal generated by transfers along proper subgroups of $A$. We then analyze the fibers of the map $\Level(A^*, \G\oplus \qz') \to \Hom(A^*,\qz')$ given by post-composition with the projection $\G\oplus \qz' \to \qz'$.

The scheme $\Level_{f^*}(A^*,\G\oplus \qz')$ admits a simple, though non-canonical, factorization. Given a decomposition $A^* \cong M \oplus K$, where $M$ is a minimal summand of $A^*$ containing $\ker(f^*)$, there is an isomorphism
\[
\Level_{f^*}(A,\G \oplus \qz') \cong \Level(M, \G)\times \Hom(K, \G).
\]
We believe that this should be viewed as a positive feature of the formal scheme $\Level_{f^*}(A,\G \oplus \qz')$ that is in contrast with the scheme $\Sub_{p^k}^{A}(\G\oplus \qz')$ studied in \cite{SSST}. The closest thing to a similar decomposition of $\Sub_{p^k}^{A}(\G\oplus \qz')$ is given in Proposition 6.5 of \cite{SSST}.

In fact, we compare $\Level_{f^*}(A^*,\G\oplus \qz')$ to the scheme $\Sub_{p^k}^{\im(f^*)}(\G \oplus \qz')$. There is a canonical map
\[
\im \colon \Level_{f^*}(A^*, \G \oplus \qz') \to \Sub_{p^k}^{\im(f^*)}(\G \oplus \qz')
\]
defined by sending a level structure $l \colon A^* \hookrightarrow \G \oplus \qz'$ such that the composite $A^* \hookrightarrow \G \oplus \qz' \to \qz'$ is equal to $f^*$ to the subgroup scheme $\im l$, which is a subgroup of $\G \oplus \qz'$ that projects onto $\im(f^*) \subset \qz'$.

Further, if $i \colon A \hookrightarrow \Sigma_{p^k}$ is the Cayley embedding (so that $|A|=p^k$) and $|A/\im f| = p^j$, then there is a topologically induced map
\[
E^0(B(\im i f \wr \Sigma_{p^j}))/I_{Tr}^{[if]} \to E^0(BA)/I_{\cF_f}
\]
making the following diagram of formal schemes commute
\[
\xymatrix{\Spf((E^0(BA)/I_{\cF_f})^{\tors}) \ar[r] \ar[d]_{\cong} & \Spf(E^0(B(\im i f \wr \Sigma_{p^j}))/I_{Tr}^{[if]}) \ar[d]^-{\cong} \\ \Level_{f^*}(A^*, \G \oplus \qz') \ar[r] & \Sub_{p^k}^{\im(f^*)}(\G \oplus \qz'),}
\]
where the right vertical map is the isomorphism of \cite[Proposition 7.12]{SSST}.

\vspace{.25in}
\paragraph{\emph{Acknowledgments}} We thank Charles Rezk for getting both of us interested in this subject. We also thank Jeremy Hahn for useful conversations. We thank Tobias Barthel, Bert Guillou, and Tomer Schlank for helpful comments on the paper. The first author is supported by the Young Scientists Fund of the National Natural Science Foundation of China (Grant No. 11901591). The second author is supported by NSF grant DMS-1906236.

\section{Recollections and extensions}\label{recol}

We recall the relationship between the $E$-cohomology of finite abelian groups and level structures on the universal deformation of a height $n$ formal group over a perfect field. Further, we note that a result of Hopkins, Kuhn, and Ravenel regarding localizations of the $E$-cohomology of groups of the form $(\Z/p^k)^n$ extend to finite abelian $p$-groups $A$ such that $A/pA$ has $p$-rank $n$.

Let $\G$ be the universal deformation of a height $n$ formal group over a perfect field of characteristic $p$. Let $E$ be the associated Morava $E$-theory. As in \cite{handbook}, we will view $\G$ as a functor from the category of complete local $E^0$-algebras to the category of abelian groups.

Let $A$ be a finite abelian $p$-group such that $A = A[p^k]$, where $A[p^k]$ is the $p^k$-torsion in $A$. The formal scheme of maps $\Hom(A,\G)$ sends a complete local $E^0$-algebra $R$ to the set of maps of commutative group schemes from $A$ to $\G[p^k]$ over $\Spf(R)$. This is isomorphic to the set of maps of abelian groups from $A$ to $\G(R)$.

Building on the canonical isomorphism $E^0(BS^1) \cong \Sect_{\G}$, Hopkins, Kuhn, and Ravenel discovered the following fundamental relationship between formal groups and Morava $E$-theory:

\begin{prop} \cite[Proposition 5.12]{hkr} Let $A$ be a finite abelian group. There is a canonical isomorphism of $E^0$-algebras $$E^0(BA)\cong \Sect_{\Hom(A^*, \G)}$$
natural in maps of finite abelian groups.
\label{HKR512}\end{prop}

Applying $\Spf(-)$, we have a canonical isomorphism of formal schemes $\Hom(A^*, \G) \cong \Spf(E^0(BA))$. Thus, given a complete local ring $R$, there is a canonical isomorphism of sets between $\Hom(A^*, \G)(R) = \Hom_{\text{Ab}}(A^*, \G(R))$ and $\Hom_{\text{cts } E^0\text{-alg}}(E^0(BA),R)$.

Given a coordinate $E^0(BS^1) \cong E^0\powser{x}$, we have $E^0(B\Z/p^k) \cong E^0\powser{x}/[p^k](x)$, where $[p^k](x)$ is the $p^k$-series for the formal group law associated to the coordinate. Further, since the Weierstrass preparation theorem implies that $E^0(B\Z/p^k)$ is a finitely generated free $E^0$-module, for $A \cong \prod\limits_{i=1}^{j} \Z/p^{k_i}$, we have
\[
E^0(BA) \cong E^0\powser{x_1, \ldots, x_j}/([p^{k_1}](x), \ldots, [p^{k_j}](x)).
\]

A map of abelian groups $l \colon A \to \G(R)$ is a level structure if the divisor $\im(l[p]) = \Spf(R\powser{x}/(\prod_{a \in A[p]} (x - l(a)))$ is a subgroup scheme of $\Spf(R) \times_{\Spf(E^0)} \G[p]$. This implies that $\im(l) = \Spf(R\powser{x}/(\prod_{a \in A} (x - l(a)))$ is a subgroup scheme of $\Spf(R) \times_{\Spf(E^0)} \G$. The formal scheme of level structures $\Level(A, \G)$ sends a complete local $E^0$-algebra $R$ to the set of level structures from $A$ to $\G(R)$.

Let $I_A \subset E^0(BA)$ be the ideal generated by the image of the transfer maps from proper subgroups of $A$.  Given a commutative ring $R$, let $R^{\tors} = \im(R \to \Q \otimes R)$. In \cite[Section 7]{ahs}, Ando, Hopkins, and Strickland produce a canonical isomorphism
\begin{equation} \label{ahsiso}
(E^0(BA)/I_A)^{\tors} \cong \Sect_{\Level(A^*, \G)}.
\end{equation}
We discuss this isomorphism more in Section \ref{main}.

Let $S_A \subset E^0(BA)$ be the set of Euler classes of nontrivial line bundles on $BA$ (ie. line bundles induced by nontrivial irreducible representations of $A$). In \cite[Lemma 6.12]{hkr}, it is noted that the canonical map $E^0(BA) \to S_{A}^{-1}E^0(BA)$ factors through $E^0(BA)/I_A$. This follows from Frobenius reciprocity together with the fact that, for each proper subgroup $A' \subset A$, there is a nontrivial representation $A \to S^1$ with $A'$ contained in the kernel. 

In \cite[Proposition 6.5]{hkr}, Hopkins, Kuhn, and Ravenel show that, when $A \cong (\Z/p^k)^n$, the canonical map $E^0(BA)/I_A \to S_{A}^{-1}E^0(BA)$ induces an isomorphism
\begin{equation} \label{hkrisolevel}
\Q \otimes E^0(BA)/I_A \cong S_{A}^{-1}E^0(BA).
\end{equation}
We will make use of a mild extension of this result. We begin with a lemma. Given a map of finite abelian groups $\rho \colon A \to H$, we will write $\rho^*(S_H) \subset E^0(BA)$ for the image of $S_H \subset E^0(BH)$ in $E^0(BA)$.

\begin{lemma} \label{lemthefirst}
Let $A$ be a finite abelian $p$-group such that $A/pA \cong (\Z/p)^n$. Let $\rho \colon A \to A/pA$ be the quotient map. Inverting the set $S_A \subset E^0(BA)$ is equivalent to inverting $\rho^*(S_{A/pA}) \subset E^0(BA)$.
\end{lemma}
\begin{proof}
By construction $\rho^*(S_{A/pA}) \subseteq S_A$. We will show that inverting $\rho^*(S_{A/pA})$ inverts $S_A$. Let $s = e(\cL) \in S_A$. Since $[p^j](s) = e(\cL^{\otimes p^j})$ is the Euler class associated to the $p^j$th tensor power of the line bundle $\cL$, there exists a $j \geq 0$ such that $[p^j](s) \in \rho^*(S_{A/pA})$. Since $s|[p^j](s)$, $s$ is a unit if $[p^j](s)$ is a unit.
\end{proof}

Let $T_A =\im(E^0(BA) \to S_{A}^{-1}E^0(BA))$.

\begin{prop} \label{hkrresult2}
Let $A$ be a finite abelian $p$-group such that $A/pA \cong (\Z/p)^n$. There is a canonical isomorphism
\[
\Q \otimes T_A \cong S_{A}^{-1}E^0(BA).
\]
\end{prop}
\begin{proof}
This follows the proof of Proposition 6.5 in \cite{hkr}. The set $\rho^*(S_{A/pA}) \subset T_A$ provides a collection of distinct roots of $\langle p \rangle (x) = [p](x)/x = p+\ldots$. Thus
\[
\prod_{s \in \rho^*(S_{A/pA})} (x-s)
\]
differs from $\langle p \rangle (x)$ by multiplication by a unit in $T_A\powser{x}$. Thus inverting $p$ is equivalent to inverting $\rho^*(S_{A/pA})$ in $T_A$ and, by Lemma \ref{lemthefirst}, this is equivalent to inverting $S_A$.
\end{proof}

\begin{cor} \label{rationaliso1}
Let $A$ be a finite abelian $p$-group such that $A/pA \cong (\Z/p)^n$. There is a canonical isomorphism
\[
\Q \otimes E^0(BA)/I_A \cong S_{A}^{-1}E^0(BA).
\]
\end{cor}
\begin{proof}
Since $\Q \otimes E^0(BA)/I_A \cong \Q \otimes \Sect_{\Level(A^*,\G)}$ and $\Sect_{\Level(A^*,\G)}$ is a domain, we have
\[
\Q \otimes E^0(BA)/I_A \cong \Q \otimes T_A \cong \Q \otimes \Sect_{\Level(A^*,\G)}.
\]
\end{proof}

Thus we see that $T_A \cong (E^0(BA)/I_A)^{\tors}$.

\section{Level structures}\label{level}

The goal of this section is to describe and study level structures on $p$-divisible groups of the form $\G \oplus \qz'$ over $\Spf(R)$, where $R$ is complete local, $\G$ is the $p$-divisible group associated to a height $n$ formal group, and $\qz ' \cong (\Q_p/\Z_p)^h$ is a height $h$ constant $p$-divisible group. Level structures on $p$-divisible groups were studied in quite a bit of generality in \cite[Chapter 3]{HarrisTaylor}. A good reference for level structures on formal groups is \cite{subgroups}.

Let $A$ be a finite abelian $p$-group and assume that $|A| = p^k$. Let $\pi \colon \G \oplus \qz' \to \qz'$ be the projection.

\begin{definition} \label{leveldef}
An $A$-level structure $l \colon A \hookrightarrow \G \oplus \qz'$ is a homomorphism of group schemes $l \colon A \to (\G \oplus \qz')[p^k]$ such that the induced map $\ker(\pi l) \rightarrow \G$ is a $\ker(\pi l)$-level structure on $\G$.
\end{definition}

There are canonical maps
\[
\Hom(A, \G \oplus \qz') \to \Hom(A,\qz')
\] 
and
\[
\Level(A,\G \oplus \qz') \to \Hom(A,\qz')
\]
given by post-composing with the projection $\pi \colon \G \oplus \qz' \to \qz'$. Fix a map $g \colon A \to \qz'$. Define
\[
\Hom_g(A,\G \oplus \qz')
\]
to be the scheme of homomorphisms $h \colon A \to \G \oplus \qz'$ such that $\pi h = g$. By construction, there is a canonical isomorphism
\[
\Hom_g(A,\G \oplus \qz') \cong \Hom(A,\G).
\]

Similarly, we define $\Level_g(A,\G\oplus \qz')$:

\begin{definition}
Given $g \colon A \to \qz'$, define
\[
\Level_g(A,\G\oplus \qz')
\]
to be the functor from complete local $R$-algebras to sets sending a complete local ring $S$ to the set of $A$-level structures on the pullback $\Spf(S) \times_{\Spf(R)} (\G\oplus \qz')$ such that post-composing with $\pi$ gives $g$.
\end{definition}

Since $\Hom(A, \qz')$ is a set, we may decompose $\Level(A,\G \oplus \qz')$ using the map $\Level(A,\G \oplus \qz') \to \Hom(A,\qz')$.

\begin{prop}
There is a decomposition of formal schemes
\[
\Level(A,\G \oplus \qz') = \Coprod{g \in \hom(\lat,A)} \Level_g(A,\G\oplus \qz').
\]
For some values of $g$, $\Level_g(A,\G\oplus \qz')$ may be the empty scheme.
\end{prop}

From the definition, we see that $\Level_g(A,\G \oplus \qz')$ is the pullback of formal schemes
\begin{equation}
\xymatrix{\Level_g(A,\G\oplus \qz') \ar[r] \ar[d] & \Level(\ker(g), \G) \ar[d]^{\iota} \\ \Hom(A, \G) \ar[r]^-{\mathrm{Res}} & \Hom(\ker(g), \G),}\label{levelfdef}
\end{equation}
Where the left vertical map is induced by the projection $\G \oplus \qz' \to \G$. This description makes $\Level_g(A,\G \oplus \qz')$ very accessible; it can be understood in terms of classical objects.

\begin{example}
Two extreme cases follow from the pullback above: If $g \colon A \to \qz'$ is injective, then $\iota$ is an isomorphism and we have
\[
\Level_g(A,\G\oplus \qz') \cong \Hom(A, \G).
\]
If $g \colon A \to \qz'$ is the zero map, then $\mathrm{Res}$ is an isomorphism and we have
\[
\Level_g(A,\G\oplus \qz') \cong \Level(A, \G).
\]
\end{example}

\begin{remark} \label{remarkinj}
Note that, as a special case of Definition \ref{leveldef}, $\Level(A, \qz')$ is the set of injective group homomorphisms from $A$ to $\qz'$.
\end{remark}

\begin{prop}
The functor $\Level_g(A,\G \oplus \qz')$ is corepresentable by an $R$-algebra that is finitely generated and free as an $R$-module and is a closed subscheme of $\Hom_g(A,\G \oplus \qz')$.
\end{prop}
\begin{proof}
We will make use of the pullback of \eqref{levelfdef}. The map 
\[
\Sect_{\Hom(\ker(g), \G)} \to \Sect_{\Level(\ker(g), \G)}
\]
is surjective and $\Sect_{\Level(\ker(g), \G)}$ is a finitely generated free $R$-module. Further, $\Sect_{\Hom(A, \G)}$ is a finitely generated free module over $\Sect_{\Hom(\ker(g), \G)}$, via the canonical inclusion $\mathrm{Res}^*$. Thus 
\[
\Sect_{\Level_g(A,\G\oplus \qz')} \cong \Sect_{\Hom(A, \G)} \otimes_{\Sect_{\Hom(\ker(g), \G)}} \Sect_{\Level(\ker(g), \G)}
\]
is a finitely generated free $R$-module.

Since $\Level(\ker(g), \G)$ is a closed subscheme of $\Hom(\ker(g), \G)$, we have $\Level_g(A,\G\oplus \qz')$ is a closed subscheme of $\Hom(A, \G) \cong \Hom_g(A, \G \oplus \qz')$. The composite
\[
\Level_g(A,\G\oplus \qz') \to \Hom(A, \G) \cong \Hom_g(A, \G \oplus \qz')
\]
is the canonical inclusion.
\end{proof}

Recall that $l \colon A \to \G$ is a level structure if and only if $l[p] \colon A[p] \to \G$ is a level structure. From the definition above, we see that if $l \colon A \hookrightarrow \G \oplus \qz'$ is a level structure and $M \subseteq A$ is a minimal summand of $A$ containing $\ker(\pi l)$, then the induced map $M \to \G$ is an $M$-level structure on $\G$. This may be compared with Part 4 of Lemma 3.1 in Harris-Taylor \cite{HarrisTaylor}. Note that they are working with full rank level structures.

The next lemma shows that $\Level_g(A,\G \oplus \qz')$ admits a simple non-canonical description. This should be viewed as a positive aspect of the scheme $\Level_g(A,\G \oplus \qz')$ that is in contrast with the scheme $\Sub_{p^k}^{A}(\G\oplus \qz')$ from \cite{SSST}. The closest thing to a decomposition of $\Sub_{p^k}^{A}(\G\oplus \qz')$ is given in Proposition 6.5 of \cite{SSST}.

%According to the discussion after Definition \ref{leveldef}, for a minimal summand $M$ of $A$ containing $\ker(g)$, the induced map $M \to \G$ is an $M$-level structure on $\G$. 
%Different choices of minimal summand $M$ are isomorphic and give equivalent level structures. We pick a minimal factorization $A= M\times K$ and have the pullback square
\begin{lemma}
Assume that $A \cong M \oplus K$ and that $M$ is a minimal summand of $A$ containing $\ker(g)$. There is an isomorphism
\[
\Level_g(A,\G \oplus \qz') \cong \Level(M, \G)\times \Hom(K, \G)
\]
depending on the choice of decomposition of $A$.
\end{lemma}
\begin{proof}
Since $M$ is a minimal summand of $A$ containing $\ker(g)$, we have $M[p] = \ker(g)[p]$. Thus we have a pullback square of formal schemes
\begin{equation}
\xymatrix{\Level_g(A,\G \oplus \qz') \ar[r] \ar[d] & \Level(M, \G) \ar[d] \\ \Hom(M \oplus K, \G) \ar[r]^-{\mathrm{Res}} & \Hom(M, \G).}\label{pblevelfm}
\end{equation}
The result follows from the fact that $\Hom(M \oplus K, \G) \cong \Hom(M,\G) \times \Hom(K,\G)$.
\end{proof}

\section{Transfer ideals}\label{transfer}

In this section we recall the relationship between transfers and the Hopkins-Kuhn-Ravenel character map and apply it to understand the quotient of $E^0(\Loops^{h}BA)$, when $A$ is a finite abelian $p$-group, by a certain transfer ideal. 

Since the free loop space of a finite cover is a finite cover, the functor $E^*(\Loops^h(-))$ has transfer maps for finite covers by making use of the fact that $E^*(-)$ is a cohomology theory.
For any subgroup $A'$ of $A$, $BA'\longrightarrow BA$ is equivalent to a finite cover and so there is a transfer map $$\Tr_{A', A}: E^0(\Loops^{h}BA')\longrightarrow E^0(\Loops^{h}BA).$$
Summing over all the proper subgroups of $A$, we get a map $$\bigoplus_{A' \subset A}\Tr_{A', A}: \bigoplus_{A' \subset A} E^0(\Loops^{h}BA') \longrightarrow E^0(\Loops^{h}BA).$$ Frobenius reciprocity implies that the image of this map is an ideal.

\begin{definition}
The transfer ideal $I_A\subset E^0(\Loops^{h}BA)$ is defined by
$$I_A= \im(\bigoplus_{A' \subset A} \Tr_{A', A}).$$\label{tranideal}
\end{definition}

\begin{remark}The formula in Definition \ref{tranideal} can be simplified to include only the maximal proper subgroups of $A$. Since the transfer map has the property that $$\Tr_{A', A}\circ \Tr_{A'', A'}
=\Tr_{A'', A}$$ for $A'' \subset A' \subset A$, $\im(\Tr_{A'', A})$ is contained in $\im(\Tr_{A', A})$.\end{remark}

In \cite[Theorem D]{hkr} Hopkins, Kuhn, and Ravenel describe the relationship between character theory and transfers. We recall this relationship now while introducing some notation.

Let $\lat = \Z_{p}^n$ and let $\qz = \lat^*$ so that $\qz \cong \QZ{n}$. Hopkins, Kuhn, and Ravenel introduce a $\Q \otimes E^0$-algebra $C_0$ that carries the universal isomorphism of $p$-divisible groups $\qz \to \G$, where
%$\qz' = (\Q_p/\Z_p)^n$ and
$\G$ is the $p$-divisible group associated to the universal deformation formal group. Explicitly, the ring $C_0$ can be taken to be
\[
C_0:= S^{-1}E^0_{cont}(B\lat),
\]
where
\[
E^0_{cont}(B\lat) = \colim_{k} E^0(B\lat/p^k\lat)
\]
and $S$ is the set of Euler classes of nontrivial irreducible representations of the abelian group $\lat/p^k\lat$ as $k$ varies.

Let $G$ be a finite group and let $\Hom(\lat,G)$ denote the set of continuous group homomorphisms from $\lat$ to $G$. Note that this set is isomorphic to the set of $n$-tuples of pairwise commuting $p$-power order elements in $G$ and that $G$ acts on this set by conjugation. Let $\Cl_n(G,C_0)$ be the ring of $C_0$-valued functions on the set of conjugacy classes in $\Hom(\lat,G)$. Hopkins, Kuhn, and Ravenel produce a character map
\[
\chi \colon E^0(BG) \lra{} \Cl_n(G,C_0)
\]
in the following way: given $[\al \colon \lat \to G]$, they map $E^0(BG)$ to $C_0$ by the composite
\[
E^0(BG) \lra{\al^*} E^0_{cont}(B\lat) \to C_0 = S^{-1}E^0_{cont}(B\lat).
\]
In Theorem C of \cite{hkr}, they prove that this map induces an isomorphism
\[
C_0 \otimes_{E^0} E^0(BG) \lra{\cong} \Cl_n(G,C_0).
\]
When $G=A$ is abelian, the conjugation action of $A$ on $\Hom(\lat,A)$ is trivial. Thus
\[
C_0 \otimes_{E^0}E^0(BA) \cong \Prod{\Hom(\lat,A)}C_0 = C_{0}^{\Hom(\lat,A)},
\]
which is the ring of $C_0$-valued functions on the set $\Hom(\lat,A)$.

Theorem D of \cite{hkr} gives a formula relating transfers in $E$-cohomology along inclusions of subgroups and the character map. Specialized to an inclusion of abelian groups $A' \subset A$, we learn that the transfer map extends a $C_0$-valued function on $\Hom(\lat,A')$ to the larger domain $\Hom(\lat,A)$ by extending by zero and multiplying by $|A/A'|$. In other words, given a function $\gamma \colon \Hom(\lat,A') \to C_0$, we have that
\begin{equation*}
\Tr_{A',A}^{C_0}(\gamma)(\al) = \left\{
\begin{array}{rl}
|A/A'|\gamma(\al) & \text{if $\al$ factors through $A'$} \\
0 & \text{otherwise}.
\end{array} \right.
\end{equation*}

Returning to the situation at hand, note that there is an equivalence $\Loops^h BA \simeq \Hom(\lat',A) \times BA$, where $\lat' = \Z_{p}^h$. Given $A' \subset A$, the induced map $\Loops^hBA' \to \Loops^hBA$ is equivalent to the cover
\[
\Hom(\lat',A') \times BA' \rightarrow \Hom(\lat',A) \times BA,
\]
where the component in the domain corresponding to $\lat' \to A'$ is sent to the component corresponding to the composite $\lat' \to A' \subset A$ in the codomain by the covering map $BA' \to BA$ induced by the inclusion $A' \subset A$.
Applying $E$-cohomology we get the factor-wise transfer map
\[
\Prod{\Hom(\lat',A')} E^0(BA') \xrightarrow{\Prod{\Hom(\lat',A')}\Tr_{A',A}^{E}} \Prod{\Hom(\lat',A')} E^0(BA)
\]
followed by the inclusion
\[
i \colon \Prod{\Hom(\lat',A')} E^0(BA) \hookrightarrow \Prod{\Hom(\lat',A)} E^0(BA).
\]
Here $\Tr_{A',A}^{E}$ is the transfer map for $E$-cohomology. 

Recall that $\qz'=(\lat')^*$.
\begin{lemma} The character map induces an isomorphism of $C_0$-algebra
\[
C_0 \otimes_{E^0} E^0(\Loops^hBA)/I_{A} \cong \Prod{\Level(A^*, \qz \oplus \qz')} C_0. \label{charlevel}
\]
\end{lemma}
\begin{proof}
Putting together the observations above, we see that applying the functor $C_0 \otimes_{E^0} E^0(-)$ to the map $\Loops^h BA' \to \Loops^h BA$ produces the map
\[
\Prod{\Hom(\lat',A')} \Prod{\Hom(\lat,A')}C_0 \xrightarrow{i \circ (\Prod{\Hom(\lat',A')}\Tr_{A',A}^{C_0})} \Prod{\Hom(\lat',A)} \Prod{\Hom(\lat,A)}C_0.
\]
By the formula for $\Tr_{A', A}^{C_0}$, this map sends the factor corresponding to a pair of maps $(\lat' \to A', \lat \to A')$ to the factor in the codomain corresponding to $(\lat' \to A' \subset A, \lat \to A' \subset A)$ by the multiplication by $|A/A'|$ map on $C_0$ (which is a $\Q$-algebra).

Since $I_{A}$ is the ideal generated by the image of the transfer maps from all proper subgroups of $A$, we wish to understand the quotient of
\[
\Prod{\Hom(\lat',A)} \Prod{\Hom(\lat,A)}C_0
\]
by the image of these transfer maps for all proper subgroups of $A$.

Thus the quotient of the product above, which is isomorphic to $C_0 \otimes_{E^0} E^0(\Loops^hBA)/I_{A}$, consists of the factors corresponding to pairs of maps $(\lat' \to A, \lat \to A)$ such that the image of both maps do not lie in the same maximal proper subgroup. In other words, the images of the maps generate all of $A$. This is equivalent to the statement that the induced map $\lat \oplus \lat' \to A$ is surjective. Thus the Pontryagin dual $A^* \to \qz \oplus \qz'$ is injective and, by Remark \ref{remarkinj}, an element in $\Level(A^*, \qz \oplus \qz')$.
\end{proof}

\section{The isomorphism}\label{main}
In this section, we give an algebro-geometric description of $E^0(\Loops^hBA)/I_A$ (after removing possible torsion) generalizing the classical case $h=0$ described in \cite[Section 7]{ahs}.

We begin by generalizing Proposition \ref{HKR512} to the iterated free loop space of $BA$. In Section \ref{level}, we constructed a map of formal schemes $\Hom(A^*, \G \oplus \qz') \lra{} \Hom(A^*, \qz')$. Further, the components of the scheme $\Spf(E^0(\Loops^hBA))$ are in bijective correspondence with $\pi_0\Map(B\lat', BA) = \Hom(\lat', A)$. Given a component, the Pontryagin dual is a homomorphism $A^* \to \qz'$. Thus there is a canonical map $\Spf(E^0(\Loops^hBA)) \to \Hom(A^*, \qz')$.

\begin{prop} \label{hkrcor}
Let $A$ be a finite abelian group. There is a canonical isomorphism of formal schemes over $\Spf(E^0)$
\[
\xymatrix{\Hom(A^*, \G \oplus \qz') \ar[rr]^-{\cong} \ar[dr] && \Spf(E^0(\Loops^hBA)) \ar[dl]\\ & \Hom(A^*, \qz') &}
\]
compatible with the maps to $\Hom(A^*, \qz')$ described above.
\end{prop}
\begin{proof}
Let $R$ be a complete local ring and assume that we are given a continuous map of $E^0$-algebras $E^0(\Loops^hBA) \to R$. Thus for each map $t \colon \lat' \to A$, we have the data of a continuous map of $E^0$-algebras $g \colon E^0(BA) \to R$. Given an element $f \colon A \to S^1$ in $A^*$, we will produce an element in $\G(R) \oplus \qz'$. Following \cite[Proposition 5.12]{hkr}, the element in $\G(R)$ is the composite
\[
E^0(BS^1) \lra{f^*} E^0(BA) \lra{g} R.
\]
The element in $\qz'$ is given by the composite
\[
\lat' \lra{t} A \lra{f} S^1.
\]
This produces the canonical map. By construction, it is compatible with the maps to $\Hom(A^*, \qz')$.
\end{proof}

As in Section \ref{recol}, for a commutative ring $R$, we will write $R^{\tors}$ for the image of $R$ in $\Q \otimes R$. Thus $R^{\tors}$ is the quotient of $R$ by the ideal of integer torsion elements in $R$. If $R$ is a finite product of Noetherian complete local rings, then $R^{\tors}$ is also a finite product of Noetherian complete local rings.

\begin{thm} \label{thmiso} There is a canonical isomorphism of formal schemes over $\Spf(E^0)$
\[
\xymatrix{\Level(A^*, \G \oplus \qz') \ar[rr]^-{\cong} \ar[dr] && \Spf((E^0(\Loops^hBA)/I_{A})^{\tors}) \ar[dl]\\ & \Hom(A^*, \qz') &}
\]
compatible with the canonical maps to $\Hom(A^*, \qz')$.

\end{thm}
\begin{proof}
We begin with the commutative diagram below:
$$\xymatrix{E^0(\Loops^{h}BA)\ar[r]^-{\cong} \ar@{->>}[d] &\mathcal{O}_{\Hom(A^*, \G\oplus\qz')} \ar@{->>}[d]\\
E^0(\Loops^{h}BA)/I_A\ar[d] &\mathcal{O}_{\Level(A^*, \G\oplus\qz')}\ar@{^{(}->}[d] \\
C_0 \otimes_{E^0} E^0(\Loops^{h}BA)/I_A\ar[r]^-{\cong} &\Prod{\Level(A^*, \qz \oplus \qz')} C_0.}$$
The top horizontal map is the isomorphism of Proposition \ref{hkrcor}. The top right vertical map is a surjection because $\Level(A^*, \G \oplus \qz')$ is a closed subscheme of $\Hom(A^*, \G \oplus \qz')$. The injection $\mathcal{O}_{\Level(A^*, \G\oplus\qz')}\hookrightarrow \Prod{\Level(A^*, \qz \oplus \qz')} C_0$ is induced by base change:
\[
\mathcal{O}_{\Level(A^*, \G\oplus\qz')}\longrightarrow C_0\otimes_{E^0}\mathcal{O}_{\Level(A^*, \G\oplus\qz')}\cong
\mathcal{O}_{\Level(A^*, \qz \oplus \qz')}\cong \Prod{\Level(A^*, \qz \oplus \qz')} C_0.
\]
It is an injection since $\mathcal{O}_{\Level(A^*, \G\oplus\qz')}$ is a finitely generated free $E^0$-module. The bottom horizontal map is the isomorphism of
Lemma \ref{charlevel}.

The bottom left map is induced by the character map. By \cite[Theorem C]{hkr}, it lands in the $\GL_n(\Z_p)$-invariants; this is $\Q \otimes E^0(\Loops^{h}BA)/I_A$. Thus we have the induced diagram
$$\xymatrix{E^0(\Loops^{h}BA)\ar[r]^-{\cong} \ar@{->>}[d] &\mathcal{O}_{\Hom(A^*, \G\oplus\qz')} \ar@{->>}[d]\\
(E^0(\Loops^{h}BA)/I_A)^{\tors} \ar@{^{(}->}[d] \ar@{-->}[r] &\mathcal{O}_{\Level(A^*, \G\oplus\qz')}\ar@{^{(}->}[d] \\
C_0 \otimes_{E^0} E^0(\Loops^{h}BA)/I_A\ar[r]^-{\cong} &\Prod{\Level(A^*, \qz \oplus \qz')} C_0.}$$
Using this diagram, we may produce the dashed arrow by choosing a lift of an element in $(E^0(\Loops^{h}BA)/I_A)^{\tors}$ and then pushing it into $\mathcal{O}_{\Level(A^*, \G\oplus\qz')}$. This is well-defined since the bottom vertical maps are injective. It is an isomorphism since we can produce an inverse by the reverse procedure.

Compatibility with the maps to $\Hom(A^*,\qz')$ follows from the fact that the isomorphism is compatible with the isomorphism of Proposition \ref{hkrcor} by construction.
\end{proof}

\begin{remark}
It would be nice to know when the canonical map $$E^0(\Loops^hBA)/I_{A} \to (E^0(\Loops^hBA)/I_{A})^{\tors}$$ is an isomorphism. It is known that this map is an isomorphism when $A$ is cyclic and is not necessarily an isomorphism when the rank of $A$ is greater than $n$.
\end{remark}

\begin{example}
We show that the map $E^0(\Loops^hBA)/I_{A} \to (E^0(\Loops^hBA)/I_{A})^{\tors}$ is not an isomorphism when $E = K_2$ is $2$-adic $K$-theory, $A = \Z/2 \times \Z/2$, and $h=0$. We thank Jeremy Hahn for pointing this out to us several years ago.

In this case, after picking the coordinate on $\hat{\G}_m$ with formal group law $x+y+xy$, we have
\[
K_{2}^{0}(BA)\cong K_{2}^{0}\powser{x, y}/([2](x), [2](y))=K_{2}^{0}\powser{x, y}/(x^2+2x, y^2+2y).
\]
Recalling \cite[Proposition 4.2]{quillencomplex}, the transfer ideal $I_{A}$ is generated
by $\langle 2 \rangle(x), \langle 2 \rangle(y)$, and $\langle 2 \rangle(x+y+xy)$, corresponding to the three maximal subgroups of $A$. Here, $x\langle 2 \rangle(x) = [2](x)$. In our case, these power series are the polynomials $x+2$, $y+2$, and $xy+x+y+2$. Since
\[
2 = (x+2) + (y+2) + (xy + x + y +2) - (x+2)(y+2),
\]
we have an equality of ideals
\[
(2, x, y) = (x+2, y+2, xy+x+y+2)
\]
so $K_{2}^{0}(BA)/I_A \cong K_{2}^{0}/(2) \cong \F_2$.
\end{example}

\section{Decomposing the transfer ideal}\label{decotran}
Let $f \colon \lat' \to A$. We will associate a family of subgroups $\cF_f$ of $A$ to the map $f$. Recall that a family of subgroups of a finite group $G$ is a set of subgroups that is closed under conjugation and taking subgroups. Thus, given a collection of subgroups of a group, we may form a family of subgroups by closing the set of subgroups under conjugation and taking subgroups. The family formed in this way is the minimal family of subgroups containing the chosen set of subgroups. We define
\[
F_f = \{H \subset A| H \text{ is proper and }f \text{   does factor through } H\}
\]
and define $\cF_f$ to be the minimal family of subgroups of $A$ containing $F_f$.

\begin{example} \label{ex1}
If $f \colon \lat' \to A$ is surjective, then $F_f = \emptyset$ and $\cF_f = \emptyset$.
\end{example}

\begin{example}\label{ex2}
Assume that $A = C_{p^k}$. If $f \colon \lat' \to C_{p^k}$ is not surjective, then $\im f = C_{p^h}$ for $h<k$. In this case, $F_f = \{C_{p^j}|h \leq j < k\}$ and $\cF_f$ is the family of all proper subgroups of $C_{p^k}$. Thus every non-surjective map from $\lat'$ to $C_{p^k}$ gives rise to the same family.
\end{example}

\begin{example}
Assume that $A = C_4 \times C_4$. If $\im f = C_2 \times \{e\} \subset C_4 \times C_4$, then the family of subgroups $\cF_f$ is the family of all proper subgroups of $A$. However, if $\im f = C_4 \times \{e\}$, then $\cF_f$ is the minimal family of subgroups containing $C_4 \times C_2 \subset C_4 \times C_4$. Therefore $\cF_f$ need not contain all of the proper subgroups of $A$.
\end{example}

\begin{example}
Assume that $A = C_{2}^{\times 3}$ and assume that $\im f = C_2 \times \{e\} \times \{e\} \subset C_{2}^{\times 3}$. Then $\cF_f$ consists of the minimal family of subgroups of $C_{2}^{\times 3}$ containing the three maximal subgroups that contain $C_2 \times \{e\} \times \{e\}$. This shows that sometimes $\cF_f$ is not determined by a single maximal subgroup of $A$.
\end{example}

Given an abelian $p$-group $A$, we will write $\cF_A$ for the family of all proper subgroups of $A$. Note that $\cF_A = \cF_0$, where $0 \colon \lat' \to A$ is the zero map. Given a map $q \colon A \to A'$ and a family $\cF$ of subgroups of $A'$, we set $q^* \cF$ to be the minimal family of proper subgroups of $A$ containing the set of subgroups $\{q^{-1}(H)|H \in \cF\}$. %The next lemma follows immediately from the third isomorphism theorem for groups and provides a description of $\cF_f$ in terms of $A/\im f$.

The next lemmas follow from the definition of the family $\cF_f$ and basic facts about finite abelian groups.

\begin{lemma}\label{basics}
Let $f \colon \lat ' \to A$, let $q \colon A \to A/\im f$ be the quotient map, and recall that $\cF_{A/\im f}$ is the family of all proper subgroups of $A/\im f$. Then
\begin{itemize}
\item $\cF_f = q^*\cF_{A/\im f}$; and
\item if $\im f \subseteq pA$, then $\cF_f = \cF_{A}$.
\end{itemize}
\end{lemma}
\begin{proof}
The first statement follows from the third isomorphism theorem. The maximal proper subgroups of $A$ that contain $\im f$ are in bijective correspondence with the maximal proper subgroups of $A/\im f$. The second statement follows from the fact that maximal proper subgroups of $A$ are in bijective correspondence with the maximal proper subgroups of $A/pA$. 
\end{proof}

\begin{lemma}\label{basics2}
Let $\pi \colon A \times B \to A$ be the projection and let $\cF$ be a family of subgroups of $A$. It follows that $\pi^* \cF$ is the family of subgroups of $A \times B$ generated by the subgroups of the form $H \times B$, where $H \in \cF$.
\end{lemma}

Let $I_{\cF_f}$ be the ideal of $E^0(BA)$ generated by the image of transfers from $H \in \cF_f$.

\begin{example}
Assume $A = C_{p^k}$. Fix a coordinate $x$ on $\G$, so that $E^0(BC_{p^k}) \cong E^0 \powser{x}/[p^k](x)$, where $[p^k](x)$ is the $p^k$-series. If $f \colon \lat ' \to C_{p^k}$ is surjective, then Example \ref{ex1} implies that $I_{\cF_f} = (0)$. However, if $f$ is not surjective, then Example \ref{ex2} implies that $\cF_f = \cF_{C_{p^k}}$ and that the ideal $I_{\cF_f} = (\langle p^k \rangle(x))$, where $\langle p^{k} \rangle (x) = [p^k](x)/[p^{k-1}](x)$.
\label{cyclictr}\end{example}

The pullback of families of subgroups and taking the quotient by the transfer ideal interact well together.

\begin{prop}
Let $q \colon A \to A'$ be a surjective map of abelian $p$-groups and let $\cF$ be a family of subgroups in $A'$. The map $q$ induces a map of $E^0$-algebras
\[
E^0(BA')/I_{\cF} \to E^0(BA)/I_{q^*\cF}.
\]
\end{prop}
\begin{proof}
The ideals $I_{\cF}$ and $I_{q^*\cF}$ are determined by the maximal subgroups in $\cF$ and $q^*\cF$. Assume that $H \in \cF$ is maximal, then $q^{-1}H \subset A$ is a maximal subgroup of $q^*\cF$ and all of the maximal subgroups of $q^*\cF$ arise in this way by the third group isomorphism theorem. There is a homotopy pullback of spaces
\[
\xymatrix{Bq^{-1}H \ar[r] \ar[d] & BH \ar[d] \\ BA \ar[r] & BA'.}
\]
This follows from the fact that homotopy pullback is the geometric realization of the ``double coset groupoid" $(A'/H) \mmod A$. Since $q$ is surjective, $A$ acts transitively on $A'/H$. The stabilizer of the coset $eH$ is precisely $q^{-1}H$.

Applying $E$-cohomology gives the commutative diagram
\[
\xymatrix{E^0(BH)\ar[r]\ar[d]^{\Tr^{E}_{H,A'}} &E^0(Bq^{-1} H)\ar[d]^{\Tr^{E}_{q^{-1}H,A}} \\ E^0(BA')\ar[r] &E^0(BA).}
\]
Frobenius reciprocity implies that the image of the transfer map is an ideal. Thus we have a well-defined map of $E^0$-algebras
\[
E^0(BA')/\im(\Tr^{E}_{H,A'}) \longrightarrow E^0(BA)/\im(\Tr^{E}_{q^{-1}H,A}).
\]
Taking the sum of the images of the transfer maps as we vary through the maximal subgroups in $\cF$, we attain the desired map.
\end{proof}

Applying this proposition to $q \colon A \to A/\im(f)$ and $\cF_{A/\im(f)}$ gives the following corollary:

\begin{cor}
The quotient map $A \to A/\im(f)$ induces a map
\[
E^0(BA/\im(f))/I_{\cF_{A/\im(f)}} \to E^0(BA)/I_{\cF_f},
\]
where $\cF_{A/\im(f)}$ is the family of all proper subgroups of $A/\im(f)$. \label{quotient}
\end{cor}

Given $f \colon \lat' \to A$ and a decomposition $A = M \oplus K$, we will write $f_M \colon \lat' \to M$ for the composite of $f$ with the projection onto $M$.

\begin{prop}\label{fdecomp}
Assume given $f \colon \lat' \to A$ and a decomposition $A = M \times K$ such that $e \times K \subset \im f$ and $\im f_M \subset pM$, then there is an isomorphism
\[
E^0(BA)/I_{\cF_f}\cong E^0(BM)/I_M\otimes_{E^0} E^0(BK)
\]
depending on the decomposition of $A$.%where $A=M\times K$ denotes a decomposition of $A$ with $M$ a minimal summand containing $\ker(f)$.
\end{prop}

\begin{proof}
It suffices to show that $\cF_f = \pi^*\cF_M$. If that is the case, then $I_{\cF_f} \subseteq E^0(BM \times K)$ is generated by the image of transfers along subgroups of the form $H \times K$, where $H$ is a proper subgroup of $M$. The result follows.

To see that $\cF_f = \pi^*\cF_M$, consider the following commutative square of abelian groups
\[
\xymatrix{M \times K \ar[r]^{\pi} \ar[d]_{q} & M \ar[d]^{q_M} \\ (M \times K)/\im f \ar[r]^{\cong} & M/\im f_M,}
\]
in which the vertical maps are both quotient maps. Lemma \ref{basics} implies that $q^* \cF_{(M \times K)/\im f} = \cF_f$. Going around the square the other way, we see that since $\im f_M \subset pM$, Lemma \ref{basics} implies that each $q_{M}^*\cF_{M/\im f_M} = \cF_M$ and Lemma \ref{basics2} implies that $\cF_f = \pi^*q_M^*\cF_{M/\im f_M} = \pi^*\cF_M$.
\end{proof}

Let $S_f = \{e(\rho \colon A \to S^1) | \ker(\rho) \in \cF_f\}$, the set of Euler classes of nontrivial irreducible representations of $A$ with kernel in the family determined by $f$. As with $\cF_f$, we will write $S_A$ for the set of Euler classes of all nontrivial irreducible representations of $A$.

\begin{example}
Assume $A = C_{p^k}$. Fix a coordinate $x$ on $\G$, so that $E^0(BC_{p^k}) \cong E^0 \powser{x}/[p^k](x)$, where $[p^k](x)$ is the $p^k$-series. If $f \colon \lat ' \to C_{p^k}$ is surjective, $S_f = \emptyset$. However, if $f$ is not surjective, then $S_f = \{x, [2](x), [3](x), \ldots, [p^{k}-1](x)\}$.
\label{cyclictr2}\end{example}

Just as in the classical case, described in Section \ref{recol}, the localization map $E^0(BA) \to S_{f}^{-1}E^0(BA)$ factors through the quotient $E^0(BA)/I_{\cF_f}$ to give a canonical map of $E^0$-algebras
\[
E^0(BA)/I_{\cF_f} \to S_{f}^{-1}E^0(BA).
\]
This follows from the fact that if $\rho \colon A \to S^1$ has the property that $\ker(\rho) \in \cF_f$ so that $e(\rho) \in S_f$, then the restriction of $e(\rho)$ to $E^0(B\ker(\rho))$ is zero. Frobenius reciprocity then implies that multiplication by $e(\rho)$ kills the image of the transfer from $\ker(\rho)$. Thus, inverting all of these Euler classes kills the entire transfer ideal.

This leads us to the following proposition relating the geometric fixed points construction in equivariant stable homotopy theory and transfer ideals of the form $I_{\cF_f}$:

\begin{prop}
There are canonical isomorphisms of $E^0$-algebras
\[
\pi_0\Phi^{\cF_f}(\underline{E}) \cong S_{f}^{-1}E^0(BA) \cong S_{f}^{-1}E^0(BA)/I_{\cF_f},
\]
where $\Phi^{\cF_f}(\underline{E})$ is the geometric fixed points for the family of subgroups $\cF_f$ of the Borel completion $\underline{E}$ of the spectrum $E$.
\label{geofixed}
\end{prop}
\begin{proof}
The first isomorphism follows from Proposition 3.20 in \cite{gmequivariant}. The second isomorphism follows from the discussion above.
\end{proof}

\begin{remark}
Note that Proposition \ref{geofixed} is true for any family $\cF$ of subgroups of $A$ and the corresponding set of Euler classes $S_{\cF}:= \{e(\rho \colon A \to S^1) | \ker(\rho) \in \cF\}$.
Not all families $\cF$ are contained in the set $\{\cF_f | f\in \hom(\lat', A)\}.$ We do not give an algebro-geometric interpretation to each $E^0$-algebra of the form $E^0(BA)/I_{\cF}$. %To study Morava E-theory of different loop spaces may provide good answers to this question.
\end{remark}

Finally, we extend Proposition \ref{hkrresult2} to families of the form $\cF_f$. Let $T_{A,f} = \im(E^0(BA) \to S_{f}^{-1}E^0(BA)$.

\begin{prop} \label{Taf}
Assume given $f \colon \lat' \to A$ and a decomposition $A \cong M \oplus K$ such that $e \times K \subset \im f$, $\im f_M \subset pM$, and $M/pM \cong (\Z/p)^n$, then there is a canonical isomorphism
\[
\Q \otimes T_{A,f} \cong S_{f}^{-1}E^0(BA).
\]
\end{prop}
\begin{proof}
Let $\pi$ be the projection $A \to M$. This induces $\pi^* \colon E^0(BM) \to E^0(BA)$. Since $e \times K \subset \im f$, we have $\pi^*\cF_{f_M} = \cF_f$ and $\pi^*S_{f_M} = S_{f}$, where $f_M = \pi f$. Thus there is an isomorphism
\[
S_{f}^{-1}E^0(BA) \cong S_{f_M}^{-1}E^0(BM) \otimes_{E^0} E^0(BK).
\]
Now since $\im f_M \subset pM$, we have $\cF_{f_M} = \cF_M$. Proposition \ref{hkrresult2} implies that
\[
S_{f_M}^{-1}E^0(BM) \cong \Q \otimes T_M.
\]
Finally, since $T_{A,f} \cong T_M \otimes_{E^0} E^0(BK)$, we have
\[
\Q \otimes T_{A,f} \cong \Q \otimes T_M \otimes_{E^0} E^0(BK) \cong S_{f_M}^{-1}E^0(BM) \otimes_{E^0} E^0(BK) \cong S_{f}^{-1}E^0(BA).
\]
\end{proof}

\begin{cor} \label{corinvert}
Under the assumptions of Proposition \ref{Taf}, there is a canonical isomorphism
\[
\Q \otimes E^0(BA)/I_{\cF_f} \cong S_{f}^{-1}E^0(BA).
\]
\end{cor} %move to next subsection?
\begin{proof}
In the proof of Corollary \ref{rationaliso1}, we showed that $\Q \otimes E^0(BM)/I_M \cong Q \otimes T_M$. Applying Proposition \ref{fdecomp}, we have 
\[
\Q \otimes E^0(BA)/I_{\cF_f} \cong \Q \otimes E^0(BM)/I_M \otimes_{E^0} E^0(BK) \cong \Q \otimes T_M \otimes_{E^0} E^0(BK)
\] 
and the result follows.
\end{proof}

\section{The fibers}\label{fiber}

We study the fibers of the canonical isomorphism in Theorem \ref{thmiso}. The Pontryagin dual of $f \colon \lat' \to A$ is a map $f^* \colon A^* \to \qz '$.

\begin{thm} There is a canonical isomorphism of $E^0$-algebras
\[
(E^0(BA)/I_{\cF_f})^{\tors} \cong \mathcal{O}_{\Level_{f^*}(A^*, \G\oplus \qz')}.
\]
\label{fmain}
\end{thm}

\begin{proof}
It suffices to understand the fibers of the isomorphism of Theorem \ref{thmiso}. By definition, the fiber over $f^* \colon A^* \to \qz'$ of the map $\Level(A^*, \G \oplus \qz') \to \Hom(A^*, \qz')$ is $\Level_{f^*}(A^*, \G \oplus \qz')$. Thus it suffices to understand the fiber over $f^* \colon A^* \to \qz'$ of the map
\[
\Spf((E^0(\Loops^hBA)/I_{A})^{\tors}) \to \Hom(A^*, \qz').
\]
This map is the composite
\[
\xymatrix{\Spf((E^0(\Loops^hBA)/I_{A})^{\tors}) \ar[d] && \\ \Spf(E^0(\Loops^hBA)) \ar[r]^-{\cong} & \Hom(\lat',A) \times \Spf(E^0(BA)) \ar[d]_{\pi} & \\  & \Hom(\lat',A) \ar[r]^-{\cong} & \Hom(A^*,\qz').}
\]
Thus we must understand the part of the ideal $I_A$ in the factor corresponding to $f \colon \lat' \to A$. In Section \ref{transfer}, we show that the ideal $I_A$ is the ideal generated by transfers along the covers
\[
\Hom(\lat',A') \times BA' \to \Hom(\lat',A) \times BA
\]
for proper subgroups $A' \subset A$. Thus, given $f \colon \lat' \to A$, we learn that the factor of $I_A$ corresponding to $f$ is the ideal of $E^0(BA)$ generated by transfers along proper subgroups $A' \subset A$ with the property that $\im f \subseteq A'$. This is precisely the ideal $I_{\cF_f} \subseteq E^0(BA)$.
\end{proof}

\begin{remark}
There is another approach to the proof of the theorem above by combining the isomorphism of Equation \eqref{ahsiso} with the decomposition of $\Level_{f^*}(A^*, \G \oplus \qz')$ coming from Equation \eqref{levelfdef} and the decomposition of $E^0(BA)/I_{\cF_f}$ in Proposition \ref{fdecomp}.
\end{remark}

Define $S \subset E^0(\Loops^hBA)$ to be the set of Euler classes of component-wise nontrivial complex line bundles over $\Loops^hBA$. 

\begin{lemma} There is a canonical isomorphism of $E^0$-algebras
\[S^{-1}E^0(\Loops^hBA)\cong \prod_{f: \lat\rightarrow A}S_f^{-1}E^0(BA).\]
 \end{lemma}
The lemma follows immediately from the definition of the sets $S$ and $S_f$.

There is an analogous result to the isomorphism of Equation \eqref{hkrisolevel} for $E^0(\Loops^hBA)$, when $A \cong (\Z/p^k)^{n+h}$. %For the Morava $E$-theory of loop space, this isomorphism is also true.
\begin{prop}
Assume $A$ is isomorphic to $(\Z/p^k)^{n+h}$. There is an isomorphism of $\Q \otimes E^0$-algebras
\[
\Q \otimes \mathcal{O}_{\Level(A^*, \G \oplus \qz')} \cong S^{-1}E^0(\Loops^hBA).
\] In addition, for each $f: \lat\rightarrow A$, there is an isomorphism \[
\Q \otimes\mathcal{O}_{\Level_{f^*}(A^*, \G\oplus \qz')} \cong S_f^{-1}E^0(BA).
\]\label{agc0}
\end{prop}

\begin{proof}
Assume that we are given $f \colon \lat' \to A$ and a level structure
\[
\xymatrix{A^* \ar[r] \ar[dr]_{f^*} & \G \oplus \qz' \ar[d] \\ & \qz',}
\]
where $A \cong (\Z/p^k)^{n+h}$.

By definition, this is equivalent to the data of the map $f^*$ and a level structure $\ker(f^*) \hookrightarrow \G$. For a level structure $\ker(f^*) \hookrightarrow \G$ to exist, $\rk(\ker(f^*)[p])$ must be less than or equal to $n$. Since $\rk(\im(f^*)[p]) \leq h$ and $\im(f^*[p]) \subset \im(f^*)[p]$, the short exact sequence of vector spaces
\[
\ker(f^*)[p] \hookrightarrow A^*[p] \twoheadrightarrow \im(f^*[p])
\]
implies that $\rk(\ker(f^*)[p]) = n$ and $\rk(\im(f^*[p])) = h$. Thus $\im(f^*[p]) = \im(f^*)[p] = \qz'[p]$.

Now Theorem \ref{fmain} together with Corollary \ref{corinvert} imply that
\[
\Q \otimes\mathcal{O}_{\Level_{f^*}(A^*, \G\oplus \qz')} \cong S_f^{-1}E^0(BA).
\]

If $\rk(\ker(f^*)[p]) > n$, then $\Level_{f^*}(A^*, \G\oplus \qz') = \emptyset$. We must also check that this holds on the right hand side. In this case $S_{f}^{-1}E^0(BA)$ receives a map of commutative rings from $S_{A/\im f}^{-1}E^0(BA/\im f)$ and, since $\ker(f^*) = (A/(\im f))^*$, we have $\rk((A/\im f)[p]) > n$. Let $Z = (A/\im f)[p]$. Since $S_{A/\im f}^{-1}E^0(BA/\im f)$ receives a map of commutative rings from $S_{Z}^{-1}E^0(BZ)$, it suffices to show that $S_{Z}^{-1}E^0(BZ) = 0$.

This follows from two facts: Each $s \in S_Z$ satisfies $[p](s) = 0$ and, if $s,t \in S_Z$, then $s-t$ is a unit. The second fact follows from the fact that $s-t$ and $s-_{\G}t$ differ by multiplication by a unit. Now, the Vandermonde matrix associated to the set $S$ (ie. the matrix in which each row consists of $\{1, s, s^2, \ldots, s^{|S_Z|-1}\}$) has determinant $\prod(s_i - s_j)$, for $s_i,s_j \in S_Z$ distinct, which is a unit. However, the $p^n$th column in the Vandermonde matrix is a linear combination of the earlier columns since each element in $S_Z$ satisfies the $p$-series. Thus $0$ is a unit in $S_{Z}^{-1}E^0(BZ)$.
\end{proof}

Recall that $\ker(f^*) = (A/(\im f))^* \subset A^*$. Corollary \ref{quotient} and Theorem \ref{fmain} are related in the following way:
\begin{prop}
There is a commutative diagram of formal schemes over $\Spf(E^0)$
\[
\xymatrix{\Spf E^0(BA)/I_{\cF_f} \ar[d] \ar[r] &   \Spf E^0(BA/\im(f))/I_{\cF_{A/\im(f)}} \ar[d] \\ \Level_{f^*}(A^*,\G\oplus \qz') \ar[r] & \Level(\ker(f^*),\G).}
\]\label{imdia}
\end{prop}

\begin{proof}

It suffices to show that the front face of the cube below commutes
$$\xymatrix{& E^0(BA/\im(f)) \ar[rr]^-{\cong}\ar'[d][dd]\ar[dl]& & \mathcal{O}_{\Hom(\ker(f^*), \G)} \ar[dd]\ar[dl]\\ (E^0(BA/\im(f))/I_{\cF_{A/\im(f)}})^{\tors} \ar[rr]^{\cong}\ar[dd]& & \mathcal{O}_{\Level(\ker(f^*), \G)} \ar[dd]&\\& E^0(BA) \ar'[r][rr]^-{\cong}\ar[dl]& &  \mathcal{O}_{\Hom (A^*, \G)}\ar[dl]\\ (E^0(BA)/ I_{\cF_{f}})^{\tors} \ar[rr]^-{\cong}& & \mathcal{O}_{\Level(A^*, \G)}.  & }$$

All of the other faces of the cube commute. The back face commutes by the naturality of the isomorphism $E^0(B(-)) \cong \Sect_{\Hom((-)^*,\G)}$ in finite abelian groups. The right face is purely algebro-geometric. The left face commutes by the construction used to produce the map in Corollary \ref{quotient}. The top and bottom squares commute by the construction of the isomorphism in Theorem \ref{fmain}.

Since the arrows from the back square to the front square are all surjective, the front square commutes as well.
\end{proof}

\section{The relation to subgroups}

In this section, we describe the relationship between the isomorphism of Theorem \ref{fmain} and the isomorphism of Corollary 7.12 in \cite{SSST}.

Assume $|A| = p^k$. Embed $A$ in $\Sigma_{p^k} = \Aut_{\text{Set}}(A)$ via the Caley embedding, so that we have $i \colon A \hookrightarrow \Sigma_{p^k}$ exhibiting $A$ as a transitive abelian subgroup of $\Sigma_{p^k}$. We have an induced map
\[
\Loops^{h}BA \to \Loops^{h}B\Sigma_{p^k}.
\]
Fix $f \colon \lat' \to A$.

Making use of the language in \cite[Section 4]{SSST}, we have the following lemma:
\begin{lemma}
The composite $if \colon \lat' \to \Sigma_{p^k}$ is monotypical.
\end{lemma}

\begin{proof}
We will view $A$ as a $\lat'$-set through $f$. To see that $if$ is monotypical, it suffices to show that $A$ is a coproduct of isomorphic transitive $\lat'$-sets. This follows from the fact that the transitive components are the cosets of $\im(if)$ in $A$ and multiplication by $ba^{-1}$ gives an isomorphism of $\lat'$-sets between the cosets $a\im(if)$ and $b \im(if)$.
\end{proof}

Since $if$ is monotypical, there is an isomorphism $C(\im if) \cong \im(if) \wr \Sigma_{p^j}$, where $p^j = |A/\im(f)|$. Recall $I_{Tr}^{[if]} \subset E^0(B\im (if) \wr \Sigma_{p^j})$ is the ideal generated by the image of transfers along $\im(if) \wr (\Sigma_l \times \Sigma_m) \to \im(if) \wr \Sigma_{p^j}$ where $l,m>0$ and $l+m = p^j$.

\begin{prop} \label{map}
The map $E^0(B\im(i f) \wr \Sigma_{p^j}) \to E^0(BA)$ induced by the inclusion $A \subseteq \im(i f) \wr \Sigma_{p^j}$ induces a map of $E^0$-algebras
\[
E^0(B\im (i f) \wr \Sigma_{p^j})/I_{Tr}^{[if]} \to E^0(BA)/I_{\cF_f}.
\]
\end{prop}

\begin{proof}

The homotopy pullback of the diagram \[\xymatrix{ &B(\im(if)\wr(\Sigma_{l}\times \Sigma_m))\ar[d] \\ BA\ar[r] &B(\im(if)\wr \Sigma_{p^j})}\]
is the disjoint union of the classifying spaces of the form of \[B(A \cap g(\im(if)\wr \Sigma_l\times \Sigma_m)g^{-1})\] with $g$ a representative of a double coset in
\[A\backslash (\im(if)\wr \Sigma_{p^j}) / (\im(if)\wr(\Sigma_l\times \Sigma_m)).\]

Since the image of the composite $A \to \im(if)\wr \Sigma_{p^j} \to \Sigma_{p^j}$ is a transitive subgroup of $\Sigma_{p^j}$ and $\Sigma_{l} \times \Sigma_m$ is a non-transitive subgroup of $\Sigma_{p^j}$, subgroups of $A$ of the form $A\cap g(\im(if)\wr (\Sigma_l\times \Sigma_m))g^{-1}$ are proper. Thus it suffices to prove that $f$ factors through subgroups of the form $g(\im(if)\wr (\Sigma_l\times \Sigma_m))g^{-1}$. This follows from the fact that $\im(f) \cong \im(if)$ is central in $\im(if) \wr \Sigma_{p^j}$ when viewed as a subgroup through the diagonal embedding. Therefore $\im(if)$ is a subgroup of $g(\im(if)\wr (\Sigma_l\times \Sigma_m))g^{-1}$ for any choice of $g$.

We may conclude that the map $E^0(B\im(i f) \wr \Sigma_{p^j}) \to E^0(BA)$ induces a map of $E^0$-algebras
\[
E^0(B\im (i f) \wr \Sigma_{p^j})/I_{Tr}^{[if]} \to E^0(BA)/I_{\cF_f}.
\]
\end{proof}

There is also a canonical map
\[
\im \colon \Level_{f^*}(A^*, \G \oplus \qz') \to \Sub_{p^k}^{\im f^*}(\G \oplus \qz')
\]
given by sending a level structure $l \colon A^* \hookrightarrow \G \oplus \qz'$ such that the composite $A^* \hookrightarrow \G \oplus \qz' \to \qz'$ is equal to $f^*$ to the subgroup scheme $\im (l)$, a subgroup of $\G \oplus \qz'$ that projects onto $\im(f^*) \subset \qz'$.

\begin{prop}
The map of Proposition \ref{map} fits into a commutative diagram of formal schemes over $\Spf(E^0)$
\[
\xymatrix{\Spf(E^0(BA)/I_{\cF_f}) \ar[r] \ar[d] & \Spf(E^0(B\im(if) \wr \Sigma_{p^j})/I_{Tr}^{[if]}) \ar[d] \\ \Level_{f^*}(A^*, \G \oplus \qz') \ar[r]^{\im} & \Sub_{p^k}^{\im(f^*)}(\G \oplus \qz').}
\]
\end{prop}

\begin{proof}

Consider the following cube:
$$\xymatrix{& E^0(B(\im (if)\wr\Sigma_{p_j}))/ I_{Tr}^{[if]}\ar[rr]^{\cong}\ar'[d][dd]\ar[dl]& & \mathcal{O}_{\Sub_{p^k}^{\im(f^*)}(\G\oplus\qz')} \ar[dd]\ar[dl]\\ (E^0(BA)/I_{\cF_{f}})^{\tors} \ar[rr]^>>>>>>>>>>>>>>>{\cong}\ar[dd]& & \mathcal{O}_{\Level_{f^*}(A^*, \G\oplus\qz')} \ar[dd]&\\& Cl_n(\im(if)\wr\Sigma_{p^j}, C_0)/I_{Tr}^{[if]} \ar'[r][rr]^>>>>>>>>>>>>>>>{\cong}\ar[dl]& &  \prod\limits_{\Sub_{p^k}^{\im(f^*)}(\qz \oplus \qz')}C_0\ar[dl]\\ Cl_n(A, C_0)/I_{\cF_{f}} \ar[rr]^{\cong}& & \prod\limits_{\Level_{f^*}(A^*, \qz\oplus\qz')}C_0. & }$$
We wish to show that the top square commutes. All of the vertical maps in the diagram are injections and are given by tensoring with $C_0$ over $E^0$. The left square commutes by Hopkins--Kuhn--Ravenel character theory. The right square commutes as $C_0 \otimes \G \cong \qz$. The back square is built by base change applied to the isomorphism of \cite{SSST} and the front square is built by base change applied to the isomorphism of Theorem \ref{fmain}. Since the vertical maps are injective and each of the other squares commute, the top square commutes.
\end{proof}

\bibliographystyle{amsalpha}
\bibliography{mybib}

\def\cprime{$'$}
\providecommand{\bysame}{\leavevmode\hbox to3em{\hrulefill}\thinspace}
\providecommand{\MR}{\relax\ifhmode\unskip\space\fi MR }
% \MRhref is called by the amsart/book/proc definition of \MR.
\providecommand{\MRhref}[2]{%
  \href{http://www.ams.org/mathscinet-getitem?mr=#1}{#2}
}
\providecommand{\href}[2]{#2}
\begin{thebibliography}{Hua18b}

\bibitem[AHS04]{ahs}
Matthew Ando, Michael~J. Hopkins, and Neil~P. Strickland, \emph{The sigma
  orientation is an {$H_\infty$} map}, Amer. J. Math. \textbf{126} (2004),
  no.~2, 247--334. \MR{2045503 (2005d:55009)}

\bibitem[BE15]{berwickevans2015twisted}
Daniel Berwick-Evans, \emph{Twisted equivariant differential k-theory from
  gauged supersymmetric mechanics}, 2015.

\bibitem[Bis85]{Bismut}
Jean-Michel Bismut, \emph{Index theorem and equivariant cohomology on the loop
  space}, Comm. Math. Phys. \textbf{98} (1985), no.~2, 213--237. \MR{786574}

\bibitem[BS16]{bscentralizers}
Tobias Barthel and Nathaniel Stapleton, \emph{Centralizers in good groups are
  good}, Algebr. Geom. Topol. \textbf{16} (2016), no.~3, 1453--1472.
  \MR{3523046}

\bibitem[Dri74]{drinfeldell1}
V.~G. Drinfel{\cprime}d, \emph{Elliptic modules}, Mat. Sb. (N.S.)
  \textbf{94(136)} (1974), 594--627, 656. \MR{0384707}

\bibitem[Gan13]{ganterpowerops}
Nora Ganter, \emph{Power operations in orbifold {T}ate {$K$}-theory}, Homology
  Homotopy Appl. \textbf{15} (2013), no.~1, 313--342. \MR{3079210}

\bibitem[GM95]{gmequivariant}
J.~P.~C. Greenlees and J.~P. May, \emph{Equivariant stable homotopy theory},
  Handbook of algebraic topology, North-Holland, Amsterdam, 1995, pp.~277--323.
  \MR{1361893}

\bibitem[Han07]{han2007supersymmetric}
Fei Han, \emph{Supersymmetric qft, super loop spaces and bismut-chern
  character}, 2007.

\bibitem[HKR00]{hkr}
Michael~J. Hopkins, Nicholas~J. Kuhn, and Douglas~C. Ravenel,
  \emph{{Generalized group characters and complex oriented cohomology
  theories.}}, J. Am. Math. Soc. \textbf{13} (2000), no.~3, 553--594 (English).

\bibitem[HT01]{HarrisTaylor}
Michael Harris and Richard Taylor, \emph{The geometry and cohomology of some
  simple {S}himura varieties}, Annals of Mathematics Studies, vol. 151,
  Princeton University Press, Princeton, NJ, 2001, With an appendix by Vladimir
  G. Berkovich. \MR{1876802}

\bibitem[Hua18a]{Huan2}
Zhen Huan, \emph{Quasi-elliptic cohomology and its power operations}, J.
  Homotopy Relat. Struct. \textbf{13} (2018), no.~4, 715--767. \MR{3870771}

\bibitem[Hua18b]{Huan1}
\bysame, \emph{Quasi-elliptic cohomology {I}}, Adv. Math. \textbf{337} (2018),
  107--138. \MR{3853046}

\bibitem[Lur19]{lurE}
Jacob Lurie, \emph{Elliptic cohomology {III}: tempered cohomology},
  \url{https://www.math.ias.edu/~lurie/papers/Elliptic-III-Tempered.pdf}, 2019.

\bibitem[NY19]{noohi2019twisted}
Behrang Noohi and Matthew~B. Young, \emph{Twisted loop transgression and higher
  jandl gerbes over finite groupoids}, 2019.

\bibitem[Qui71]{quillencomplex}
Daniel Quillen, \emph{Elementary proofs of some results of cobordism theory
  using {S}teenrod operations}, Advances in Math. \textbf{7} (1971), 29--56
  (1971). \MR{290382}

\bibitem[Rez]{rezkpowercalc}
Charles Rezk, \emph{Power operations for {M}orava {E}-theory of height 2 at the
  prime 2}, \url{arXiv:math.AT/0812.1320}.

\bibitem[SS15]{SSST}
Tomer~M. Schlank and Nathaniel Stapleton, \emph{A transchromatic proof of
  {S}trickland's theorem}, Adv. Math. \textbf{285} (2015), 1415--1447.
  \MR{3406531}

\bibitem[Sta]{handbook}
Nathaniel Stapleton, \emph{Lubin-tate theory, character theory, and power
  operations}, Handbook of homotopy theory, CRC Press/Chapman and Hall
  Handbooks in Mathematics, pp.~891--930.

\bibitem[Sta13]{tgcm}
Nathaniel Stapleton, \emph{Transchromatic generalized character maps}, Algebr.
  Geom. Topol. \textbf{13} (2013), no.~1, 171--203. \MR{3031640}

\bibitem[Sta15]{ttcm}
\bysame, \emph{Transchromatic twisted character maps}, J. Homotopy Relat.
  Struct. \textbf{10} (2015), no.~1, 29--61. \MR{3313634}

\bibitem[Str97]{subgroups}
Neil~P. Strickland, \emph{Finite subgroups of formal groups}, J. Pure Appl.
  Algebra \textbf{121} (1997), no.~2, 161--208. \MR{MR1473889 (98k:14065)}

\bibitem[Str98]{etheorysym}
N.~P. Strickland, \emph{Morava {$E$}-theory of symmetric groups}, Topology
  \textbf{37} (1998), no.~4, 757--779. \MR{1607736 (99e:55008)}

\bibitem[Zhu14]{zhupo}
Yifei Zhu, \emph{The power operation structure on {M}orava {$E$}-theory of
  height 2 at the prime 3}, Algebr. Geom. Topol. \textbf{14} (2014), no.~2,
  953--977. \MR{3160608}

\end{thebibliography}

\end{document}